\documentclass{amsart}
\usepackage{latexsym,amsxtra,amscd,ifthen}
\usepackage{amsfonts}
\usepackage{verbatim}
\usepackage{amsmath}
\usepackage{amsthm}
\usepackage{amssymb}
\usepackage{url}
\usepackage[arrow,matrix,all]{xy}
\usepackage[all]{xy}
\usepackage{tikz-cd}

\newtheorem{theorem}{Theorem}[section]

\newtheorem{corollary}[theorem]{Corollary}
\newtheorem{lemma}[theorem]{Lemma}

\newtheorem{proposition}[theorem]{Proposition}
\newtheorem{prop-def}[theorem]{Proposition-Definition}
\theoremstyle{definition}
\newtheorem{definition}[theorem]{Definition}
\newtheorem{remark}[theorem]{Remark}

\newtheorem{example}[theorem]{Example}

\numberwithin{equation}{section}

\def\op{{\rm op}}

\def\deg{{\rm deg}}

\def\Ext{{\rm Ext}}

\def\Hom{{\rm Hom}}

\def\N{\mathbb{N}}

\def\Z{\mathbb{Z}}

\def\1{\mathbbold{1}}

\def\GKdim{{\rm GKdim}\,}

\def\CM{{\rm CM}}

\def\gldim{{\rm gldim}\,}

\def\GrMod{\operatorname{GrMod}}

\def\grmod{\operatorname{grmod}}

\def\QGr{\operatorname{QGr}}
\def\qgr{\operatorname{qgr}}
\def\injdim{\operatorname{injdim}}
\def\End{\operatorname{End}}

\def\depth{\operatorname{depth}}

\begin{document}
\title[Maximal Cohen-Macaulay modules]
{Maximal Cohen-Macaulay modules over \\
a noncommutative 2-dimensional singularity}

\author{
X.-S. Qin, Y.-H. Wang and J.J. Zhang}

\address{Qin: School of Mathematical Sciences, 
Shanghai Center for Mathematical Sciences,
Fudan University, Shanghai 200433, China}

\email{13110840002@fudan.edu.cn}

\address{Wang: School of Mathematics,
Shanghai Key Laboratory of Financial Information Technology,
Shanghai University of Finance and Economics, Shanghai 200433, China}

\email{yhw@mail.shufe.edu.cn}

\address{Zhang: Department of Mathematics,
Box 354350,
University of Washington,
Seattle, WA 98195, USA}
\email{zhang@math.washington.edu}




\begin{abstract}
We study properties of graded maximal Cohen-Macaulay modules over an 
${\mathbb N}$-graded locally finite, Auslander Gorenstein, and 
Cohen-Macaulay algebra of dimension two. As a consequence, we extend 
a part of the McKay correspondence in dimension two to a more 
general setting.
\end{abstract}

\maketitle

\dedicatory{}%
\commby{}%

\setcounter{section}{-1}

\section{Introduction}
A noncommutative version of the McKay correspondence in dimension 
two was developed in \cite{CKWZ1, CKWZ2, CKWZ3}. One of the main 
ingredients was the study of the invariant subrings of connected 
graded, noetherian, Artin-Schelter regular algebras of global 
dimension two under natural actions of quantum binary polyhedral 
groups. The McKay quivers \cite[Definition 2.9]{CKWZ3} of these 
quantum binary polyhedral groups are twisted versions of 
$\widetilde{A}\widetilde{D}\widetilde{E}$ graphs where the details
can be found in \cite[Proposition 7.1]{CKWZ1}. It was proved in 
\cite[Theorem B]{CKWZ3} that the McKay quiver is isomorphic to the 
Gabriel quiver \cite[Definition 2.8]{CKWZ3} of the smash product 
algebra corresponding to the action. 

The noncommutative singularities (or equivalently, their associated 
algebras) studied in \cite{CKWZ1, CKWZ2, CKWZ3} are usually far from 
commutative and do not satisfy a polynomial identity. For these 
noncommutative singularities we introduced the concept of a 
noncommutative quasi-resolution \cite[Definition 0.5]{QWZ} which 
generalizes Van den Bergh's noncommutative crepant resolution 
\cite{VdB1, VdB2}. The smash product constructions used in 
\cite{CKWZ1, CKWZ2, CKWZ3} are examples of noncommutative 
quasi-resolutions. Recently Reyes-Rogalski proved that the Gabriel 
quivers of non-connected, ${\mathbb N}$-graded, Artin-Schelter regular 
algebras of global dimension two are twisted versions 
(which are called {\it pretzeled quivers} in this paper) of the
$\widetilde{A}\widetilde{D}\widetilde{E}$ graphs. 

Recent study of the invariant theory of (non-connected graded) 
preprojective algebras under finite group actions initiated by 
Weispfenning \cite{We} suggests that one should extend the 
noncommutative McKay correspondence to a larger class of not 
necessarily connected, graded algebras. The aim of this short 
paper is to supply a small piece of the puzzle in this slightly 
more general version of the noncommutative McKay correspondence. 
Let $\Bbbk$ be a base field and let MCM (respectively, CM) 
stand for ``maximal Cohen-Macaulay'' (respectively, 
``Cohen-Macaulay''). We summarize the main results as follows.

\begin{theorem}
\label{xxthm0.1}
Let $A$ be a noetherian ${\mathbb N}$-graded locally finite 
algebra of Gelfand-Kirillov dimension two. Suppose that
\begin{enumerate}
\item[(a)]
$A$ has a balanced dualizing complex,
\item[(b)]
$A$ is Auslander Gorenstein and CM, and 
\item[(c)]
$A$ has a noncommutative quasi-resolution $B$.
\end{enumerate} 
Then 
\begin{enumerate}
\item[(1)]
$A$ is of finite Cohen-Macaulay type in the graded sense.
\item[(2)]
There is a one-to-one correspondence between the set of 
indecomposable MCM graded right $A$-modules up to degree shifts 
and isomorphisms and the set of graded simple right $B$-modules
up to degree shifts and isomorphisms.
\item[(3)]
Let $\{M_1,\cdots,M_d\}$ be a complete list of the 
indecomposable MCM graded right $A$-modules up to degree shifts 
and isomorphisms. Then, for some choice of integers 
$w_1,\cdots,w_d$, $C:=\End_{A}(\bigoplus_{i=1}^d M_i(w_i))$ 
is an ${\mathbb N}$-graded noncommutative quasi-resolution of $A$. 
As a consequence, $C$ is graded Morita equivalent to $B$. 
\item[(4)]
$A$ is a noncommutative graded isolated singularity.
\end{enumerate}
\end{theorem}

One should compare Theorem \ref{xxthm0.1} with results in 
\cite[Section 3]{He} in the commutative case and 
\cite[Theorem 2.5]{Jo}, \cite[Corollary 4.5 and Theorem 5.7]{CKWZ3} 
in the noncommutative case.

\begin{theorem}
\label{xxthm0.2}
Let $A$ satisfy the hypotheses in 
Theorem \ref{xxthm0.1}. Further assume that the noncommutative 
quasi-resolution $B$ in Theorem \ref{xxthm0.1}(c) is standard in the 
sense of Definition \ref{xxdef1.3}. Then the following holds.
\begin{enumerate}
\item[(1)]
The Gabriel quiver of $B$ is a pretzeled quiver of a finite 
union of graphs of $\widetilde{A}\widetilde{D}\widetilde{E}$ type. 
\item[(2)]
The standard noncommutative quasi-resolution $B$ is unique
up to isomorphism.
\end{enumerate}
\end{theorem}
 
The above theorem confirms that a generalized version of the 
noncommutative McKay correspondence should still be within the 
framework of $\widetilde{A}\widetilde{D}\widetilde{E}$ diagrams. 
Since the Gabriel quiver ${\mathcal G}(B)$ of $B$ is defined by 
using simple modules over $B$, by the correspondence given in 
Theorem \ref{xxthm0.1}(3) and the uniqueness in Theorem 
\ref{xxthm0.2}(2), ${\mathcal G}(B)$ (if it exists) is also an 
invariant of MCM modules over $A$. The proof of Theorem 
\ref{xxthm0.2} follows from Theorem \ref{xxthm0.1} and results of 
Reyes-Rogalski when one relates the results in \cite{RR1, RR2} with
the concept of noncommutative quasi-resolutions.

Terminology used in the above theorems will be explained in 
later sections. The proofs of Theorems \ref{xxthm0.1} and 
\ref{xxthm0.2} will be given in Section 3.

\subsection*{Acknowledgments} 
The authors would like to thank Ken Brown, Daniel Rogalski, Robert 
Won and Quanshui Wu for many useful conversations and valuable 
comments on the subject. X.-S. Qin was partially supported by the 
Foundation of China Scholarship Council (Grant No. [2016]3100). 
Y.-H. Wang was partially supported by the Natural Science Foundation 
of China (No. 11871071) and Foundation of China Scholarship Council 
(Grant No. [2016]3009).
Y.H. Wang and X.-S. Qin thank the Department of Mathematics, University 
of Washington for its very supportive hospitality during their 
visits. J.J. Zhang was partially supported by the US National 
Science Foundation (Grant No. DMS-1700825). 

\section{Definitions and preliminaries}
\label{xxxsec1}

Throughout let $\Bbbk$ be a field. All algebras and modules are 
over $\Bbbk$. Recall that a $\Bbbk$-algebra $A$ is {\it $\N$-graded}
if $A=\bigoplus_{n\in\N}A_n$ as vector spaces with $1\in A_0$ and 
$A_iA_j\subseteq A_{i+j}$ for all $i,j\in\N$. We say that $A$ is 
{\it locally finite} if $\dim_{\Bbbk}A_n<\infty$ for all $n$. In 
this paper, a graded algebra usually means $\N$-graded.
A right $A$-module $M$ is $\Z$-graded if $M=\bigoplus_{n\in\Z}M_n$ 
with $M_iA_j\subseteq M_{i+j}$ for all $i,j\in {\mathbb Z}$. We 
write $\GrMod A$ for the category of right graded $A$-modules with 
morphisms being the degree preserving homomorphisms, and $\grmod A$ 
for the subcategory of finitely generated right $A$-modules. Other
definitions such as degree shift or grading shift $(w)$ can be 
found in \cite{RR1, RR2}.

\subsection{Generalized Artin-Schelter regular algebras}
\label{xxsec1.1}

In this subsection, we review the definition of a generalized 
Artin-Schelter (AS) regular algebra.

\begin{definition}\cite[Definition 1.4]{RR1}
\label{xxdef1.1}
Let $A$ be a locally finite graded algebra and $J:=J(A)$ be the 
graded Jacobson radical, that is, the intersection of all graded 
maximal right ideals of $A$. Write $S=A/J$. We say that $A$ is 
{\it generalized AS Gorenstein of dimension $d$} if $A$ has 
graded injective dimension $d$ and there is a graded invertible 
$(S,S)$-bimodule $V$ such that
\begin{equation}
\label{E1.1.1}\tag{E1.1.1}
\Ext_A^i(S,A)\cong \left\{
\begin{array}{ll}
 V, & \mbox{if $i=d$},\\
 0, & \mbox{if $i\neq d$},
\end{array} \right.
\end{equation} 
as $(S,S)$-bimodules.
If further $A$ has graded global dimension $d$, then $A$ is called
{\it generalized AS regular of dimension $d$}.
\end{definition}

%
%

\begin{definition}\cite[Section 3]{RR2}
\label{xxdef1.2}
Let $A$ be a locally finite $\N$-graded $\Bbbk$-algebra. If the 
finite-dimensional semisimple algebra $S:=A/J(A)\cong A_0/J(A_0)$ 
is isomorphic to a product $\Bbbk^{\oplus d}$ of finitely many copies of 
the base field $\Bbbk$, then we say that $A$ is {\it elementary}.
\end{definition}

\begin{definition}
\label{xxdef1.3}
Let $A$ be a graded algebra. We say $A$ is {\it standard}
if $A_0$ is $\Bbbk^{\oplus d}$ for some positive integer $d$ 
and $A$ is generated by $A_0$ and $A_1$. 
\end{definition}

A very nice result proven by Reyes-Rogalski is the following.

\begin{theorem} \cite{RR2}
\label{xxthm1.4}
Let $A$ be an ${\mathbb N}$-graded generalized AS 
regular algebra. Suppose that
\begin{enumerate}
\item[(a)]
$A$ is standard,
\item[(b)]
$A$ is noetherian,
\item[(c)]
$\gldim A=2$.
\end{enumerate}
Then $A$ is isomorphic to the algebra $A_2(Q,\tau)$ 
described in \cite[Definition 7.5]{RR2} where $Q$
is a quiver whose arrows all have weight 1 and whose 
spectral radius $\rho(Q)$ is 2.
\end{theorem}

\begin{proof} We combine some results of Reyes-Rogalski.
By \cite[Theorem 1.5]{RR1}, $A$ is twisted 
Calabi-Yau in the sense of \cite[Definition 1.2]{RR1}.
By \cite[p.37]{RR2}, every locally finite elementary 
graded twisted Calabi-Yau algebra of global dimension 2 
is isomorphic to $A_2(Q,\tau)$. Since $A$ is standard, 
the weight of every arrow in  $Q$ is 1. Since $A$ is 
noetherian, by \cite[Theorem 7.8(2)]{RR2}, $\rho(Q)=2$.
\end{proof}

\subsection{Noncommutative quasi-resolutions}
\label{xxsec1.2}
In this subsection we review some definitions about 
noncommutative quasi-resolutions from \cite{QWZ}. We assume
that all algebras are noetherian in this subsection.
First we recall the definition of Gelfand-Kirillov dimension.

\begin{definition}\cite[Definition 2.1]{KL}
\label{xxdef1.5}
Let $A$ be an algebra and  $M$ a right $A$-module.
\begin{enumerate}
\item[(1)] The {\it Gelfand-Kirillov dimension} (or {\it $\GKdim$})
 of $A$ is defined to be
$$
\GKdim A=\sup_V\{\varlimsup_{n\rightarrow\infty}
\log_n(\dim V^n)\mid V\subseteq A\}
$$
where $V$ ranges over all finite dimensional $\Bbbk$-subspaces of $A$.
\item[(2)] The {\it Gelfand-Kirillov dimension} (or {\it $\GKdim$}) 
of $M$ is defined to be
$$
\GKdim M=\sup_{V,W}\{\varlimsup_{n\rightarrow\infty}\log_n(\dim WV^n)\mid
V\subseteq A,W\subseteq M\}
$$
where $V$ and $W$ range over all finite dimensional $\Bbbk$-subspaces 
of $A$ and $M$ respectively.
\end{enumerate}
\end{definition}

If $M$ is a finitely generated graded module over a locally finite
graded algebra $A$, then its $\GKdim$ can be computed by \cite[(E7)]{Zh}
$$\GKdim M=\varlimsup_{n\rightarrow\infty}\log_n\sum\limits_{j\leq n}\dim(M_j).$$
By \cite[Theorem 6.14]{KL}, $\GKdim$ is exact for modules over a 
locally finite graded algebra.

For simplicity, we always work with the dimension function $\GKdim$ in 
this paper. Next we specialize some definitions in \cite{QWZ} to the 
$\GKdim$ case and omit the prefix ``$\GKdim$'' in some cases.

\begin{definition}\cite[Definition 1.5]{QWZ}
\label{xxdef1.6}
Let $n\geq 0$. Let $A$ and $B$ be two locally finite graded algebras.
\begin{enumerate}
\item[(1)]
Two $\Z$-graded right $A$-modules $X,Y$ are called 
{\it $n$-isomorphic}, denoted by $X\cong_n Y$, if there
exist a $\Z$-graded right $A$-module $P$ and homogeneous 
homomorphisms of degree zero $f: X\to P$ and $g: Y\to P$ such 
that both the kernel and cokernel of $f$ and $g$ are in 
$\GrMod_n A$.
\item[(2)]
Two $\Z$-graded $(B,A)$-bimodules $X,Y$ are called
{\it $n$-isomorphic}, denoted by $X\cong_n Y$, if there
exist a $\Z$-graded $(B,A)$-bimodule $P$ and homogeneous 
bimodule homomorphisms with degree zero $f: X\to P$
and $g: Y\to P$ such that both the kernel and cokernel of
$f$ and $g$ are in $\GrMod_n A$ when viewed as graded right 
$A$-modules.
\end{enumerate}
\end{definition}

\begin{definition}\cite[Definitions 1.2 and 2.1]{Le}
\label{xxdef1.7}
Let $A$ be an algebra and $M$ a right $A$-module.
\begin{enumerate}
\item[(1)]
The {\it grade number} of $M$ is defined to be
$$j_A(M):=\inf\{i|\Ext_A^i(M,A)\neq0\}\in \N\cup\{+\infty\}.$$ 
If no confusion can arise, we write $j(M)$ for $j_A(M)$. Note that 
$j_A(0)=+\infty$.
\item[(2)] 
We say $M$ satisfies the {\it Auslander condition} if for any $q\geq0,$
$j_A(N)\geq q$ for all left $A$-submodules $N$ of $\Ext_A^q(M,A)$.
\item[(3)] 
We say $A$ is {\it Auslander-Gorenstein} (respectively, 
{\it Auslander regular}) of dimension $n$ if 
$\injdim A_A=\injdim {_AA}=n<\infty$ (respectively, $\gldim A=n<\infty$) 
and every finitely generated left and right $A$-module satisfies 
the Auslander condition.
\end{enumerate}
A graded version of an Auslander-Gorenstein (respectively, Auslander 
regular) algebra is defined similarly.
\end{definition}

\begin{definition} \cite[Definition 0.4]{ASZ}
\label{xxdef1.8}
Let $A$ be a locally finite graded algebra. We say $A$ is {\it graded 
Cohen-Macaulay} (or, {\it graded CM} in short) if $\GKdim(A)=d\in\N,$ 
and
$$j(M)+\GKdim(M)=\GKdim(A)$$
for every graded finitely generated nonzero left (or right) $A$-module $M$.
\end{definition}

Let $A$ be a locally finite $\N$-graded  algebra and $n$ a nonnegative 
integer. Let $\GrMod_n A$ denote the full subcategory of $\GrMod A$
consisting of $\Z$-graded right $A$-modules $M$ with $\GKdim(M)\leq n$.
Since $\GKdim$ is exact over a noetherian locally finite $\N$-graded 
algebra \cite[Theorem 6.14]{KL}, $\GrMod_n A$ is a Serre subcategory
of $\GrMod A$. Hence it makes sense to define the quotient categories
$$\QGr_nA:=\frac{\GrMod A}{\GrMod_nA} \quad \text{and} \quad
\qgr_nA:=\frac{\grmod A}{\grmod_nA}.$$
We denote the natural and exact projection functor by
\begin{equation}
\label{E1.8.1}\tag{E1.8.1}
\pi:\GrMod A\longrightarrow\QGr_n A.
\end{equation}
For $M\in\GrMod A,$ we write $\mathcal{M}$ for the object $\pi(M)$
in $\QGr_n A$. The hom-sets in the quotient category are defined by
\begin{equation}
\label{E1.8.2}\tag{E1.8.2}
\Hom_{\QGr_n A}(\mathcal{M},\mathcal{N})=
\lim_{\longrightarrow}\Hom_A(M',N')
\end{equation}
for $M$, $N$ $\in\GrMod A$, where $M'$ is a graded submodule of $M$ 
such that the $\GKdim$ of $M/M'$ is no more than $n$, $N'=N/T$ for 
some graded submodule $T\subseteq N$ with $\GKdim(T)\leq n$, and 
where the direct limit runs over all the pairs $(M',N')$ with these 
properties. Note that $\pi$ is also defined from 
\begin{equation}
\label{E1.8.3}\tag{E1.8.3}
\grmod A\longrightarrow\qgr_n A.
\end{equation} 
It follows from \cite[(E1.10.1)]{QWZ} that the functor $\pi$ 
in \eqref{E1.8.1} has a right adjoint functor
\begin{equation}
\label{E1.8.4}\tag{E1.8.4}
\omega: \QGr_n A\longrightarrow\GrMod A.
\end{equation} 
By \cite[Proposition 2.10(2)]{QWZ}, when $M$ is $(n+2)$-pure 
in the sense of Definition \ref{xxdef2.8}(6) (in the next 
section) over an Auslander Gorenstein and CM algebra $A$, 
then $\omega(\pi(M))$ agrees with the Gabber closure 
defined in \cite[Definition 2.8]{QWZ}.


We will use $\qgr_0 A$ in section 2, 
which will be denoted by $\qgr A$.

Let $\mathcal{A}$ be a category consisting of ($\N$-)graded, 
locally finite, noetherian algebras with finite $\GKdim$ 
\cite[Example 3.1]{QWZ}. Our definition of a noncommutative 
quasi-resolution will be made inside the category $\mathcal{A}$.

\begin{definition}\cite[Definition 0.5]{QWZ}
\label{xxdef1.9}
Let $A\in\mathcal{A}$ with $\GKdim(A)=d\geq 2$. 
If there are a graded Auslander-regular and $\CM$ 
algebra $B\in\mathcal{A}$ with $\GKdim(B)=d$ and two 
${\mathbb Z}$-graded bimodules $_{B}M_{A}$ and $_{A}N_{B}$, 
finitely generated on both sides, such that
$$M\otimes_{A} N\cong_{d-2} B, \quad {\text{and}}\quad
N\otimes_{B} M\cong_{d-2} A$$
as ${\mathbb Z}$-graded bimodules, then the triple $(B,M,N)$ or
simply the algebra $B$ is called a {\it noncommutative
quasi-resolution} (or {\it NQR} for short) of $A$.
\end{definition}

If $A\in\mathcal{A}$ with $\GKdim A=2$, then by 
\cite[Theorem 4.2 and Lemma 8.2]{QWZ},
any two NQRs of $A$ are graded Morita equivalent, namely, 
there is a unique noncommutative quasi-resolution of $A$ in 
the sense of Morita equivalent.


\begin{lemma}
\label{xxlem1.10}
If $A$ is noetherian, graded, locally finite, Auslander 
regular and CM, then $A$ is generalized AS regular.
\end{lemma}

\begin{proof} By the Auslander and CM properties, for every 
finite dimensional graded right $A$-module $M$, 
$$\Ext^i_A(M,A)=\begin{cases} 0 & i\neq d:=\gldim A,\\
N&i=d
\end{cases}
$$
for some finite dimensional graded left $A$-module $N$. 
By the double-Ext spectral sequence \cite[(E2.13.1)]{QWZ},
$\Ext^d_A(-,A)$ induces a bijection from the set of graded 
simple right $A$-modules up to isomorphism to the set of 
graded simple left $A$-modules up to isomorphism. By
\cite[Theorem 5.2]{RR1}, $A$ is generalized AS regular.
\end{proof}

\subsection{Other concepts}
\label{xxsec1.3}
We recall some other concepts that are used in the main theorems.
The following definition is due to Ueyama.

\begin{definition}\cite[Definition 2.2]{Ue}
\label{xxdef1.11}
Let $A$ be a noetherian graded algebra. We say $A$ is a 
{\it graded isolated singularity} if the associated noncommutative 
projective scheme $\QGr A$ has finite global dimension. 
\end{definition}

Ueyama gave this definition for connected graded algebras, but 
we consider possibly non-connected graded algebras. This concept 
is used in Theorem \ref{xxthm0.1}(4).

We will also use some results about balanced dualizing complex over 
noncommutative rings introduced by Yekutieli \cite{Ye}. We refer to 
\cite{Ye, VdB3, CWZ} for more details. We need the following local
duality formula. 

\begin{theorem} \cite{Ye, VdB3}
\label{xxthm1.12} 
Let $A$ be a noetherian, graded, locally finite algebra with 
balanced dualizing complex
$R$ and let $M$ be a graded right $A$-module. Then 
$$R\Gamma_{\mathfrak m}(M)^{\ast}={\mathrm{RHom}}_A(M,R)$$
where $\mathfrak m$ is the graded Jacobson ideal of $A$ and 
$(-)^{\ast}$ denotes the graded $\Bbbk$-linear dual.
\end{theorem}

The following corollary is well-known.

\begin{corollary}
\label{xxcor1.13} 
Let $A$ be a noetherian, graded, locally finite algebra with 
balanced dualizing complex $R$. If $A$ is generalized AS Gorenstein
of injective dimension $d$, then $R$ is of the form $\Omega(d)$ 
where $\Omega$ is a graded invertible $A$-bimodule. 
\end{corollary}

\section{Preparations}
\label{xxsec2}

In this section we will give the relation between
Gabriel quivers, pretzeled quivers of graphs, and 
noncommutative quasi-resolutions (NQRs). 

\subsection{Gabriel Quivers}
\label{xxsec2.1}
Let $Q$ be a quiver that has finitely many vertices and arrows. 
If we label the vertices of $Q$ by integers from $1$ to $d$, 
then the adjacency matrix of $Q$ is a square $d\times d$-matrix 
over ${\mathbb N}$. It is clear that there is a one-to-one 
correspondence between 
\begin{equation}
\label{E2.0.1}\tag{E2.0.1}
\{{\text{quivers with vertices labeled $\{1,2,\cdots,d\}$}}\}
\Longleftrightarrow
\{{\text{$d\times d$-matrices over ${\mathbb N}$.}}\}
\end{equation}
For this reason, the adjacency matrix of $Q$ is also denoted by 
$Q$ if no confusion occurs. The opposite quiver of $Q$ is 
obtained by changing the direction of each arrow in $Q$. Hence 
the adjacency matrix of the opposite quiver of $Q$ is the 
transpose of the adjacency matrix of $Q$. We denote the opposite 
quiver of $Q$ by $Q^{op}$. 

Next we review the definition of a Gabriel quiver. Suppose that 
$A$ is locally finite graded and elementary with 
$S=A/J(A)=\Bbbk^{\oplus d}$. One can lift (not necessarily uniquely) the 
$d$ primitive orthogonal idempotents of $S$ to an orthogonal 
family of primitive idempotents with $1=e_1+\cdots+e_n$ in 
$A_0$, see \cite[Corollary 21.32]{Lam}.  Notice that 
$A=\bigoplus_{i=1}^ne_iA$, so every $e_iA$ is a graded 
projective right $A$-module, which is an indecomposable module 
since $e_i$ is a primitive idempotent. Then we get $d$ distinct 
simple right $A$-modules $S_i:=e_iS=e_iA/e_iJ(A)$, and every 
simple graded right module is isomorphic to a shift of one of 
the $S_i$ for some $1\leq i\leq d$.

\begin{definition}
\label{xxdef2.1}
Let $A$ be an elementary, locally finite, graded algebra such that
$A/J(A)=\Bbbk^{\oplus d}$. The {\it Gabriel quiver 
$\mathcal{G}(A)$} of $A$ is defined by
\begin{enumerate}
\item[$\bullet$] 
vertices: graded simple right $A$-modules $S_1,\ldots,S_d$
corresponding to individual projection to $\Bbbk$.
\item[$\bullet$] 
arrows: $S_i\xrightarrow {q_{ij}}S_j$ if $q_{ij}=
\dim_{\Bbbk}\Ext_{A}^1(S_j,S_i)_{-1}$ where $\Ext_{A}^1(S_j,S_i)$
has a natural ${\mathbb Z}$-grading.
\end{enumerate}
Under the identification in \eqref{E2.0.1}, $\mathcal{G}(A)
=(q_{ij})_{d\times d}$ where $q_{ij}$ is defined above.
\end{definition}

Note that if $A$ in the above definition is Koszul in the sense of 
\cite[Definition 1.5]{MV}, then $\Ext_{A}^1(S_j,S_i)$ is 
concentrated in degree $-1$. In this case, 
$$q_{ij}=\dim_{\Bbbk}\Ext_{A}^1(S_j,S_i)_{-1}=
\dim_{\Bbbk} \Ext_{A}^1(S_j,S_i).$$

\begin{remark}
\label{xxrem2.2}
Suppose $A_0=\Bbbk^{\oplus d}$.
\begin{enumerate}
\item[(1)]
Let $P_i=e_i A$. Then $\{P_1,\cdots,P_d\}$ is a complete
list of indecomposable graded projective right $A$-modules
up to degree shifts and isomorphisms.
\item[(2)]
It is easy to see that
$$q_{ij}=\dim_{\Bbbk} \Hom_{A}(P_i, P_j)_1$$
for all $1\leq i,j\leq d$.
\end{enumerate}
\end{remark}

\subsection{Twists of a quiver and Pretzelizations}
\label{xxsec2.2}
Let $Q=(q_{ij})_{d\times d}$ be a quiver with $d$ vertices (or a 
$d\times d$-matrix over ${\mathbb N}$). A {\it graph} is a class 
of special quivers $Q$ with $q_{ij}=q_{ji}$ for all $i,j$.
Then a quiver $Q$ is a graph if and only if $Q=Q^{op}$. A graph is also
called a {\it symmetric quiver}. The example given in \eqref{E2.6.1} 
is a graph, and the example given in \eqref{E2.6.2} is not a graph in 
this paper.

Let $\sigma$ be an automorphism of the vertex set  $\{1,\cdots,d\}$. 
Then $\sigma$ induces an automorphism of $Q$ if and only if 
$q_{\sigma(i)\sigma(j)}=q_{ij}$ for all $i,j$. 

\begin{definition}
\label{xxdef2.3} \cite{BQWZ}
Let $Q:=(q_{ij})$ be a quiver and $\sigma$ be an automorphism of $Q$. 
The {\it twist} of $Q$ associated to $\sigma$, denoted by 
$^\sigma Q$, is the quiver corresponding to the matrix 
$P_{\sigma} Q (=QP_{\sigma})$ where $P_{\sigma}$ is the 
permutation matrix associated to $\sigma$; in other words,
$(^\sigma Q)_{ij}=q_{\sigma(i)j}=q_{i\sigma^{-1}(j)}$ for 
all $i,j$.
\end{definition}

Let $G$ be a graph (or symmetric quiver) and $\sigma$ be an
automorphism of the quiver $G$. Let $Q$ be the twisted quiver
$^{\sigma} G$. Then $Q^{op}={^{\sigma^{-2}} Q}$, which follows
from the following linear algebra computation
\begin{equation}
\label{E2.3.1}\tag{E2.3.1}
Q^{op}=(P_{\sigma} G)^{op}=(GP_{\sigma})^{op}
=P_{\sigma}^{-1} G^{op}=P_{{\sigma}^{-2}} Q={^{\sigma^{-2}} Q}.
\end{equation}

\begin{definition} \cite{BQWZ}
\label{xxdef2.4} 
Let $G$ be a graph. 
\begin{enumerate}
\item[(1)]
A quiver $Q$ is called a {\it pretzelization} of $G$ or a 
{\it pretzeled} quiver of $G$ if $Q\cup Q$ is a twisted quiver of 
a finite disjoint union of $G$. It is possible that $Q$ itself 
is a twisted quiver of another finite disjoint union of $G$. 
In general, a pretzelization of a graph is not a graph.
\item[(2)]
We say a graph $G$ is of {\it $\widetilde{A}\widetilde{D}\widetilde{E}$ 
type} if it is of type 
$$ \widetilde{A}_n, 
\widetilde{D}_n, \widetilde{L}_n, \widetilde{DL}_n,
\widetilde{E}_6, \widetilde{E}_7, \widetilde{E}_8$$
listed in \cite[Theorem 2]{HPR}. 
\end{enumerate}
\end{definition}

Assuming the readers are familiar with $\widetilde{A}
\widetilde{D}\widetilde{E}$ graphs which can be found in 
many papers including \cite{HPR}. We only list 
$\widetilde{L}_n$ for $n\geq 0$ and $\widetilde{DL}_n$ 
for $n\geq 2$ here:
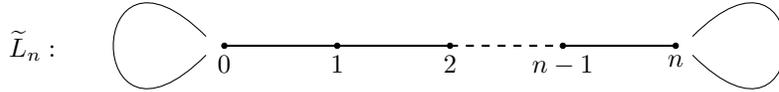
\begin{eqnarray*}
&&\begin{tikzpicture}
\draw[black,thick] (0,0) -- (1.5,0);
\draw[black, thick] (1.5,0) -- (3,0);
\draw[dashed][black, thick] (3,0) -- (4.5,0);
\draw[black, thick] (4.5,0) -- (6,0);
\filldraw[black] (0,0) circle (1pt) node[anchor=west] [below]{0};
\filldraw[black] (1.5,0) circle (1pt) node[anchor=west] [below]{1};
\filldraw[black] (3,0) circle (1pt) node[anchor=west][below] {2};
\filldraw[black] (4.5,0) circle (1pt) node[anchor=west][below] {$n-1$};
\filldraw[black] (6,0) circle (1pt) node[anchor=west] [below]{$n$};
\node (C1a) at (-0.12,0)  {};
\draw[-]  (C1a) edge [in=225,out=135,loop,looseness=24] (C1a)node[left] {$\widetilde{L}_{n}:$\ \ \ \ \ \ \ \ \ \ \ \ \ \ \ \ \ \  } ;
\node (C1b) at (6.1,0)  {};
\draw[-]  (C1b) edge [in=45,out=315,loop,looseness=24] (C1b)node[left] { } ;
\end{tikzpicture}\\
&&
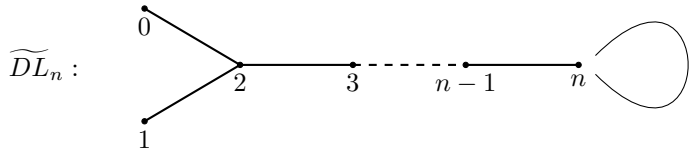
\begin{tikzpicture}
\draw[black,thick] (0.23,0.75) -- (1.5,0);
\draw[black,thick] (0.23,-0.75) -- (1.5,0);
\draw[black, thick] (1.5,0) -- (3,0)node[left] {$\widetilde{DL}_n:$\ \ \ \ \ \ \ \ \ \ \ \ \ \ \ \ \ \ \ \ \ \ \ \ \ \ \  \ \ \ \  } ;
\draw[dashed][black, thick] (3,0) -- (4.5,0);
\draw[black, thick] (4.5,0) -- (6,0);
\filldraw[black] (0.23,0.75) circle (1pt) node[anchor=west] [below]{0};
\filldraw[black] (0.23,-0.75) circle (1pt) node[anchor=west] [below]{1};
\filldraw[black] (1.5,0) circle (1pt) node[anchor=west] [below]{2};
\filldraw[black] (3,0) circle (1pt) node[anchor=west][below] {3};
\filldraw[black] (4.5,0) circle (1pt) node[anchor=west][below] {$n-1$};
\filldraw[black] (6,0) circle (1pt) node[anchor=west] [below]{$n$};
\node (C1b) at (6.1,0)  {};
\draw[-]  (C1b) edge [in=45,out=315,loop,looseness=24] (C1b)node[left] { } ;
\end{tikzpicture}
\end{eqnarray*}

In the above, the vertex set is $\{0,1,\cdots,n\}$. If there is 
an edge between vertices $i$ and $j$, then both $q_{ij}$
and $q_{ji}$ are 1. If there is a loop at vertex $i$, then $q_{ii}=1$. 


A result of Smith \cite{Sm} states that a simple graph $G$ (i.e., a graph 
with no double edges and loops) has spectral radius 2 if and only if it 
is either $\widetilde{A}_n$, $\widetilde{D}_n$, $\widetilde{E}_6$, 
$\widetilde{E}_7$ or $\widetilde{E}_8$, see also \cite[Theorem 1.3]{DGF}. 
The following result is a folklore for some experts, which follows 
from the proof of \cite[Theorem 2]{HPR}. 

\begin{lemma}
\label{xxlem2.5}
Let $G$ be a graph {\rm{(}}i.e., a symmetric quiver{\rm{)}} which is 
not necessarily simple. Then $\rho(G)=2$ if and only if $G$ is listed 
in Definition \ref{xxdef2.4}(2).
\end{lemma}

This result was also proved by Chen-Kirkman-Walton-Zhang 
when they were working on the project \cite{CKWZ1, CKWZ2, CKWZ3}, 
but it was removed by the authors in the final published version of 
\cite{CKWZ1, CKWZ2, CKWZ3}. Another related result in \cite{BQWZ}
is that graphs of types $\widetilde{L}_n$ and $\widetilde{DL}_n$ are 
pretzelizations of graphs of types $\widetilde{A}_n$ and $\widetilde{D}_n$.
Note that graphs of type $\widetilde{DL}_n$ and $\widetilde{L}_1$
appeared in \cite[Proposition7.1]{CKWZ1}. A key lemma concerning 
the pretzelization is the following.

\begin{lemma} \cite{BQWZ}
\label{xxlem2.6}
Let $Q$ be a quiver. Then $Q$ is a pretzelization of a graph 
if and only if $Q^{op}={^\mu Q}$ for some automorphism $\mu$
of the quiver $Q$.
\end{lemma}

The automorphism $\mu$ in the above lemma is called a 
{\it Nakayama automorphism} of the quiver $Q$. For example, 
$\sigma^{-2}$ is a Nakayama automorphism of $Q$ in 
\eqref{E2.3.1}.

For example, if $Q$ is the Dynkin graph of type $A_3$, 
\begin{equation}
\label{E2.6.1}\tag{E2.6.1}
\end{equation}
\begin{center}
\begin{tikzpicture}
\draw[->][black] (0.1,0.1) -- (1.4,0.1);
\draw[<-][black] (0.1,-0.1) -- (1.4,-0.1);
\draw[->][black] (1.6,0.1) -- (2.9,0.1);
\draw[<-][black] (1.6,-0.1) -- (2.9,-0.1);
\filldraw[black] (0,0) circle (1pt) node[anchor=west] {};
\filldraw[black] (1.5,0) circle (1pt) node[anchor=west] {};
\filldraw[black] (3,0) circle (1pt) node[anchor=west] {};
\end{tikzpicture}
\end{center}

\noindent
then one twist of the disconnected graph $Q\cup Q\cup Q$ 
(or a pretzelization of the graph $Q$) is the following 
connected pretzel-shaped quiver:

\begin{equation}
\label{E2.6.2}\tag{E2.6.2}
\end{equation}
{\tiny {\bf \begin{equation}
\notag
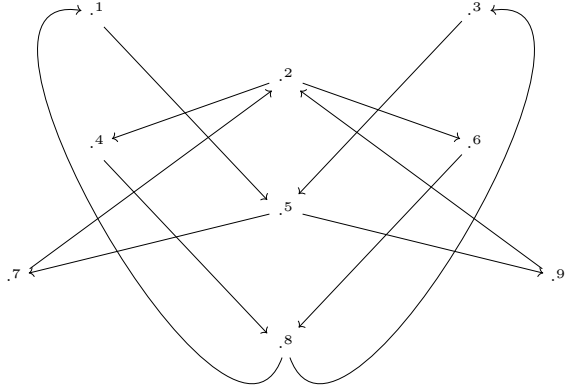
\begin{tikzcd}
&\cdot^1 \arrow[dddrrr,shift right=.6ex]
&&&&&&\cdot^3\arrow[dddlll,shift right=.6ex]
\\
&&&&\cdot^2\arrow[dlll, shift right=.4ex]\arrow[drrr,shift left=.4ex]
\\
&\cdot^4\arrow[dddrrr,shift right=.6ex] 
&&&&&&\cdot^6\arrow[dddlll,shift right=.6ex]
\\
&&&&\cdot^5 \arrow[dllll, shift right=.5ex]\arrow[drrrr,shift left=.5ex]
\\
\cdot^7\arrow[uuurrrr,bend left=0]
 &&&&&&&&\cdot^9\arrow[uuullll,bend right=0]
  \\
&&&&\cdot^8\arrow[uuuuulll,bend left=130] \arrow[uuuuurrr,bend right=130] 
\end{tikzcd}
\end{equation}}}

\begin{proposition}
\label{xxpro2.7}
Let $Q$ be a quiver such that $Q^{op}={^\sigma Q}$ for some 
automorphism $\sigma$ of $Q$. If $\rho(Q)=2$, then 
$Q$ is a pretzelization of a graph of $\widetilde{A}
\widetilde{D}\widetilde{E}$ types.
\end{proposition}

\begin{proof} 
By Lemma \ref{xxlem2.6}, $Q$ is a pretzelization of a graph
$G$. This means that $Q\cup Q={^\tau G}$ for some automorphism
$\tau$ of $G$. By a result of 
\cite{BQWZ}, $\rho$ is stable under twists of quiver. Since 
$\rho(Q)=2$, $\rho(G)=\rho(Q\cup Q)=\rho(Q)=2$. By \cite{BQWZ}
(also see \cite[Theorem 2]{HPR}),
$G$ is of $\widetilde{A}\widetilde{D}\widetilde{E}$ type in the
sense of Definition \ref{xxdef2.4}(2).
\end{proof}

\subsection{Depth and Maximal Cohen-Macaulay Modules}
\label{xxsec2.3}

We first collect some definitions from the literature. In the 
next definition we only consider graded right modules. The
same definition can be made for graded left modules. Since
we always consider graded algebras and graded modules, we 
sometimes omit the word ``graded''. The following definition
of a Cohn-Macaulay module is different from the concept of a 
Cohen-Macaulay algebra given in Definition \ref{xxdef1.8}.

\begin{definition}
\label{xxdef2.8}
Let $A$ be a noetherian, locally finite, graded algebra with finite 
GKdimension. Let $S=A/J(A)$. Let $M$ be a nonzero finitely generated 
graded right $A$-module.
\begin{enumerate}
\item[(1)]
The {\it depth} of $M$ is defined to be 
\begin{eqnarray*}
\depth_A M := \inf\{i|\Ext_A^i(S,M)\neq 0\} \in {\mathbb N}\cup\{+\infty\}.
\end{eqnarray*}
If no confusion can arise, we write $\depth M$ for $\depth_A M$.
\item[(2)]
We say $M$ is {\it Cohen-Macaulay} if $\depth M=\GKdim M$.
\item[(3)]
We say $M$ is a {\it maximal Cohen-Macaulay} (or {\it MCM} 
for short) module if $M$ is a Cohen-Macaulay and $\GKdim M=\GKdim A$.
\item[(4)]
We say $A$ is of {\it finite Cohen-Macaulay type} (in the graded sense)
if there are only finitely many graded MCM modules up to degree shifts
and isomorphisms.
\item[(5)]
$M$ is called {\it reflexive} if the natural map $M\to \Hom_{A^{op}}(\Hom_{A}(M,A),A)$
is an isomorphism.
\item[(6)]
Let $n$ be an integer. Then $M$ is called {\it $n$-pure} if 
$\GKdim N=n$ for every nonzero submodule $N\subseteq M$. 
\end{enumerate} 
\end{definition}

Note that the definition of an $n$-pure module in 
\cite[Definition 2.1(2)]{QWZ} is different from and related to ours. 
Some basic lemmas about the depth can be found in \cite[Section 5]{QWZ}.
The following lemma is clear. An object ${\mathcal M}$ in
$\qgr A$ is called 2-pure if ${\mathcal M}=\pi(M)$ for a 2-pure 
graded $A$-module $M$ and there is no nonzero sub-object 
${\mathcal N}\subseteq {\mathcal M}$ such that ${\mathcal N}
=\pi(N)$ for some $N\in \grmod A$ of GKdimension 1, where
$\pi$ is defined in \eqref{E1.8.1}.

\begin{lemma}
\label{xxlem2.9}
Let $A$ be a noetherian, locally finite, graded algebra with
$\GKdim A=2$. Suppose that $A$ is Auslander-Gorenstein and CM.
\begin{enumerate}
\item[(1)]
There is a bijection between $2$-pure objects in
$\qgr A$ and reflexive modules in $\grmod A$.
\item[(2)]
The functors $\pi$ and $\omega$ defined in 
\eqref{E1.8.3}-\eqref{E1.8.4} induce an equivalence between the
category of $2$-pure objects in $\qgr A$ and that of reflexive 
modules in $\grmod A$.
\end{enumerate}
\end{lemma}

\begin{proof} We only prove part (2) as part (1) follows immediately 
from part (2).

(2) For every reflexive module $M$, by \cite[Proposition 2.14]{QWZ},
it is 2-pure (the definition of $n$-pure is slightly
different from the definition in this paper). 
Then $\pi(M)$ is 2-pure in $\qgr A$. Conversely, 
let ${\mathcal M}$ be a 2-pure object in $\qgr A$. Let $M$
be any finitely generated 2-pure module such that ${\mathcal M}
=\pi(M)$. Let $\widetilde{M}$ be the Gabber closure of $M$ 
defined in \cite[Definition 2.8]{QWZ}. Since two such $M$ 
differ only by finite dimensional vector spaces, $\widetilde{M}$
is independent of the choices of $M$. Or equivalently, 
$\widetilde{M}$ is only dependent on ${\mathcal M}$,
which is $\omega({\mathcal M})$. Therefore 
$$\pi\omega({\mathcal M})={\mathcal M}$$
for 2-pure objects in $\qgr A$ and
$$\omega\pi(M)=M$$
for reflexive objects in $\grmod A$. The assertion follows.
\end{proof}

Here is the main result in this subsection.

\begin{theorem}
\label{xxthm2.10} 
Let $A$ be a noetherian graded locally finite algebra with
$\GKdim$ $2$ and let $M$ be a finitely generated graded $A$-module. 
Suppose that 
\begin{enumerate}
\item[(a)]
$A$ has a balanced dualizing complex $R$, and 
\item[(b)]
$A$ is Auslander-Gorenstein and CM.
\end{enumerate}
Then $M$ is reflexive if and only if it is MCM in $\grmod A$. 
\end{theorem}

\begin{proof} Since $A$ is Auslander Gorenstein and CM, 
the depth of $A$ (and $A^{op}$) is 2. It follows from 
\cite[Lemma 5.6]{QWZ} that the depth of a nonzero
reflexive module is 2. Therefore a reflexive module is MCM
as $\GKdim A=2$.

Conversely, let $M$ be a MCM right $A$-module. Then 
$\Ext^i_A(S,M)=0$ for $i=0,1$ where $S=A/J(A)$. This implies that
$R^i\Gamma_{\mathfrak m}(M)=0$ for $i=0,1$. By 
Theorem \ref{xxthm1.12}, $\Ext^{-i}_A(M,R)=0$ for
$i=0,1$. Since $R=\Omega[2]$ where $\Omega$ is a graded
invertible $A$-bimodule (see Corollary \ref{xxcor1.13}) 
and $[2]$ denotes the second complex shift,
$$\Ext^{j}_A(M,A) =\Ext^{j}_A(M,\Omega)\otimes \Omega^{-1}=0$$
for $j=1,2$. By the double-Ext spectral sequence 
\cite[(E2.13.1)]{QWZ}, $M$ is reflexive.
\end{proof}

\section{Proof of the main results}
\label{xxsec3}

We give the proofs of the main results here.

\begin{proof}[Proof of Theorem \ref{xxthm0.1}]
The statements works for left and right modules.
We only prove the results for right modules.

(1) By Lemma \ref{xxlem2.9}(1), there is a bijection between 
2-pure objects in $\qgr A$ and reflexive modules in 
$\grmod A$. Similarly, there is a bijection between 
2-pure objects in $\qgr B$ and reflexive modules in 
$\grmod B$. By \cite[Lemma 3.5]{QWZ}, $\qgr A$ is
equivalent to $\qgr B$. Therefore there is a bijection between 
reflexive modules in $\grmod A$ and those in $\grmod B$.
By Theorem \ref{xxthm2.10}, the reflexive modules in 
$\grmod A$ are exactly the MCM modules in $\grmod A$. Since 
$B$ is generalized AS regular, the reflexive modules in 
$\grmod B$ are precisely the projective modules in 
$\grmod B$. Since $B$ is locally finite, there are only 
finitely many indecomposable graded projective modules
over $B$ up to degree shifts and isomorphisms. This implies
that there are only finitely many indecomposable graded MCM 
modules over $A$ up to degree shifts and isomorphisms. 
The assertion follows by Definition \ref{xxdef2.8}(4).

(2) By the proof of part (1), there is a one-to-one 
correspondence between the set of the indecomposable
MCM graded right $B$-modules up to degree shifts and 
isomorphisms and the set of graded indecomposable projective 
right $B$-modules up to degree shifts and isomorphisms. The 
assertion follows from the fact that there is a one-to-one
correspondence between the set of the indecomposable
graded projective right $B$-modules up to degree shifts 
and isomorphisms and the set of graded simple right $B$-modules
up to degree shifts and isomorphisms.

(3) 
Let $F: \qgr B\to \qgr A$ be the equivalence given in
\cite[Lemma 3.5]{QWZ}. Let $(\pi_A, \omega_A)$ be the 
adjoint pair of functors given in \eqref{E1.8.3} and
\eqref{E1.8.4}. Similarly for $(\pi_B, \omega_B)$.
Then we have a functor
\begin{equation}
\label{E3.0.1}\tag{E3.0.1}
\Phi:= \pi_A\circ F\circ\pi_B: \grmod B\to \grmod A
\end{equation}
which is an equivalence of categories when restricted 
to the categories of reflexive modules over $A$ and $B$. 

Let $B=\bigoplus_{i=1}^d P_i^{\oplus u_i}$ where $u_i\geq 1$
and $\{P_1,\cdots,P_d\}$ is a complete list of indecomposable
projective right $B$-modules which are direct summands of $B$.
Let $C=\End_{B}(\bigoplus_{i=1}^d P_i)$. Then $C$ is
graded Morita equivalent to $B$. Let $M_i=\Phi(P_i)$.
Then $\{M_1,\cdots, M_d\}$ is a complete list of 
MCM right modules over $A$ up to degree shifts 
and isomorphisms. Since $\Phi$ is an equivalence, 
$C\cong \End_A(\bigoplus_{i=1}^d M_i)$ as desired.

(4) The assertion follows by the definition and from the fact 
that $\qgr A$ is equivalent to $\qgr B$ \cite[Lemma 3.5]{QWZ}
and that the global dimension of $\qgr B$ is bounded by the
global dimension of $\grmod B$.
\end{proof}

Recall that an algebra $A$ is called {\it indecomposable} if it 
cannot be written as a sum of two nontrivial algebras; this 
is equivalent to $A$ having no nontrivial central idempotents. 
Similarly, there is a definition of {\it graded indecomposable}
algebra. By \cite[Lemma 2.7]{RR1}, a graded algebra $A$ is 
indecomposable if and only if $A$ is graded indecomposable.

\begin{lemma}
\label{xxlem3.1}
Let $B_1$ and $B_2$ be two NQRs of a noetherian graded locally 
finite algebra $A$ of GKdimension two. Suppose
that $(B_1)_0=\Bbbk^{\oplus d_1}$ and $(B_2)_0=\Bbbk^{\oplus d_2}$ 
and that $B_1$ is standard. Then $B_1\cong B_2$ as graded algebras.
\end{lemma}

\begin{proof}
By Theorem \ref{xxthm1.4}, $B_1$ is isomorphic to 
$A_2(Q,\tau)$ for some quiver $Q$ satisfying the extra conditions 
listed in Theorem \ref{xxthm1.4}. By \cite[Lemma 3.4(2)]{RR2},
the quiver $Q$ agrees with the Gabriel quiver of $B_1$ 
[Definition \ref{xxdef2.1}]. By \cite[Theorem 0.6(1)]{QWZ}, 
$B_1$ and $B_2$ are graded Morita equivalent. Let
$$\Psi: \grmod B_1\to \grmod B_2$$
be an equivalence of categories. Via $\Psi$, one sees that
$B_1$ and $B_2$ have the same number of indecomposable 
summands and $\Psi$ matches up these indecomposable summands
as graded Morita equivalences. Therefore, without loss of generality,
one can assume that both $B_1$ and $B_2$ are indecomposable.

By Theorem \ref{xxthm0.1}(2), the number of graded simple right 
$B_1$-modules (up to degree shifts and isomorphisms) is the same 
as the number of graded simple right $B_2$-modules (up to 
degree shifts and isomorphisms). That number is $d_1=d_2=:d$ as 
$$(B_1)_0=\Bbbk^{\oplus d_1}, \quad (B_2)_0=\Bbbk^{\oplus d_2}.$$ 
Let $\{e_i\}_{i=1}^d$ be the set of primitive idempotents of 
$B_1$ (and of $B_2$). Then $\{P_i:=e_i B_1\}_{i=1}^d$ is a 
complete set of indecomposable graded projective right 
$B_1$-modules up to degree shifts and isomorphisms. Similarly, 
$\{R_i:=e_i B_2\}_{i=1}^d$ is a complete set of 
indecomposable graded projective right $B_2$-modules
up to degree shifts and isomorphisms. Since $\Psi$ is an 
equivalence and the degree shifts are also equivalences,
we may assume that $\Psi(P_i)=R_i(w_i)$ for some integer
$w_i\geq 0$ with one of $w_i$ being 0. We can further assume 
that all $w_i$ are non-negative and $w_1=0$ after some permutation.

We claim that $w_i=0$ for all $i$. If not, there is an $i$
such that $w_i>0$. By \cite[Lemma 7.3]{RR2}, the quiver 
$Q$ in $B_1=A_2(Q,\tau)$ is strongly connected, 
that is, given any two vertices $i$ and $j$ in $Q$ 
there is a directed path from $i$ to $j$. In particular,
there is a path from $1$ to $i$ in $Q$ where $w_1=0$
and $w_i>0$. Along this path, choose two vertices $a\neq b$ 
such that there is an arrow from $a$ to $b$ and $w_a=0$ 
and $w_b>0$. By the above choice, 
$$\begin{aligned}
\Hom_{B_2}(R_b(w_b), R_a(w_a))_{1}&=
\Hom_{B_2}(R_b(w_b), R_a)_{1}\\
&=\Hom_{B_2}(R_b(w_b-1), R_a)_{0}=0
\end{aligned}$$
as $w_b-1\geq 0$ and $a\neq b$. Applying $\Psi^{-1}$, we obtain
that
$$\Hom_{B_1}(P_a, P_b)_{1}=0$$
which implies that there is no arrow from 
$a$ to $b$ in the Gabriel quiver $Q$. This yields a contradiction.
Therefore all $w_i=0$ and $\Psi(P_i)=R_i$ 
for all $i$. Since $\Psi$ is an equivalence, 
we have an isomorphism of algebras
$$B_1\cong\End_{B_1}(\bigoplus_i P_i)
\cong \End_{B_2}(\bigoplus_i R_i)\cong B_2$$
as desired.
\end{proof}

\begin{proof}[Proof of Theorem \ref{xxthm0.2}]
(1) By Lemma \ref{xxlem1.10}, any standard NQR $B$ is generalized
AS regular. By Theorem \ref{xxthm1.4}, $B$ is isomorphic
to $A_2(Q,\tau)$ that is given in \cite[Definition 7.5]{RR2}.

By Theorem \ref{xxthm1.4}, the arrows in the quiver $Q$ have 
weight $1$. Note that $B$ is a direct sum of finitely
many indecomposable algebras. Without loss of generality
we may assume that $B$ is indecomposable. Then $Q$ is strongly 
connected by \cite[Lemma 7.3]{RR2}. By Theorem \ref{xxthm1.4}, 
$\rho(Q)=2$. By \cite[Theorem 1.2(2)]{RR2}, there is an automorphism 
$\mu$ of $Q$ such that $Q^{\op}= {{^{\mu}}Q}$. By 
Proposition \ref{xxpro2.7}, $Q$ is a pretzelization of a graph of 
type $\widetilde{A}\widetilde{D}\widetilde{E}$.

By \cite[Lemma 3.4]{RR2} and the definition of Gabriel quiver,
$Q$ is exactly the Gabriel quiver $\mathcal{G}(B)$ of $B$, whence,
$\mathcal{G}(B)$ is a pretzelization of a graph of 
type $\widetilde{A}\widetilde{D}\widetilde{E}$.

(2) Suppose there are two standard NQRs, say $B_1$ and $B_2$, of 
$A$. The assertion follows from Lemma \ref{xxlem3.1}.
\end{proof}

\begin{definition}
\label{xxdef3.2}
Suppose $A$ satisfies the hypotheses of Theorem \ref{xxthm0.2}.
By Theorem \ref{xxthm0.2}(2), the standard NQR of $A$ is 
unique up to isomorphism. In this case, the {\it Auslander-Reiten 
quiver} of $A$ is defined to be the Gabriel quiver of the 
standard NQR of $A$.
\end{definition}

Unfortunately, not every algebra $A$ in Theorem \ref{xxthm0.1} 
has a standard NQR, as the next example shows.

\begin{example}
\label{xxex3.3}
Let $\Bbbk[x,y]$ be a commutative polynomial ring with 
$\deg\;  x>0$ and $\deg\;  y>0$.
Let $A=\Bbbk[x,y]^{\langle\sigma\rangle}$ where $\sigma$
is the automorphism of $\Bbbk[x,y]$ of order 2 defined by
$$\sigma: x\to -x, \quad y\to -y.$$ 
Then $B=\Bbbk[x,y]\ast \langle\sigma\rangle$ is a NQR of 
$A$ by \cite[Example 8.5]{QWZ}. It is well known by the 
commutative theory that $A$ has two MCMs: $A$ itself and 
the module $C$ such that $A\oplus C=\Bbbk[x,y]$. 
\begin{enumerate}
\item[(1)]
If $\deg\;  x=\deg\;  y=1$, then $B$ is a standard NQR and is 
the preprojective algebra associated to the Dynkin graph
$\widetilde{A_1}$.
Let $B'=\End_{A}(A\oplus C(1))$. Then $B'$ is another 
NQR of $A$ and is isomorphic 
to $B$ as ungraded algebras. As an ${\mathbb N}$-graded
algebra, $B'_{0}=\begin{pmatrix} 
\Bbbk &\Bbbk\oplus \Bbbk \\0& \Bbbk\end{pmatrix}$, which is 
not semisimple. 
\item[(2)]
If $\deg\;  x>1$ or $\deg\;  y>1$, then $B$ is not standard.
Note that $B_0=\Bbbk^2$. If $A$ has a standard NQR, 
say $B'$, then by Lemma \ref{xxlem3.1}, $B\cong B'$
as graded algebras. This implies that
$B$ is standard, a contradiction. Therefore $A$ does not
have a standard NQR. 
\item[(3)]
The uniqueness of $B$ fails in Theorem \ref{xxthm0.2}(2)
if we only require $B_0=\Bbbk^{\oplus d}$. To see this,
we consider the case when $\deg\;  x=\deg\;  y=2$. It is
easy to see that elements in $A$ live in degrees $4{\mathbb N}$
and elements in $C(1)$ live in degrees $4{\mathbb N}+1$.
Let $B'=\End_{A}(A\oplus C(1))$. Then $B'$ is another NQR 
of $A$ and $B'_0=\Bbbk^2=B_0$. But $B'_1=\Bbbk^{\oplus 2}$ 
and $B_1=0$. Therefore $B' \not\cong B$. 
\end{enumerate}
\end{example}

The existence of standard NQRs can be proved in the following
case.

\begin{lemma}
\label{xxlem3.4}
Let $R$ be a standard, noetherian, graded, locally finite 
Auslander Gorenstein and CM algebra with $\GKdim\geq 2$. 
Let $H$ be a semisimple Hopf algebra acting on $R$ homogeneously 
and inner-faithfully. Assume that the homological determinant
of the $H$-action \cite[Definition 3.7]{RRZ} is trivial. 
Let $A=R^H$. Suppose that $\Bbbk$ 
is algebraically closed and that the conditions in 
\cite[Example 8.5]{QWZ} hold. Then $A$ has a standard NQR. 
As a consequence, the Auslander-Reiten quiver of $A$ is a 
pretzelization of a graph of $\widetilde{A}
\widetilde{D}\widetilde{E}$ type.
\end{lemma}

\begin{proof} Let $B=A\# H$ as in \cite[Example 8.5]{QWZ}.
By \cite[Example 8.5]{QWZ}, $B$ is a NQR of $A$. It is 
easy to check that $B_0$ is semisimple and $B$ is generated 
by $B_0$ and $B_1$. Write $B=\bigoplus_{i=1}^d P_i^{\oplus u_i}$ 
where $u_i\geq 1$ and $\{P_1,\cdots,P_d\}$ is a complete list 
of indecomposable projective right $B$-modules which are
direct summands of $B$. Let $C=\End_{B}(\bigoplus_{i=1}^d P_i)$. 
Then $C$ is graded Morita equivalent to $B$. As a consequence,
$C$ is a NQR of $A$. By working with the minimal
projective resolution of the graded simples, one can 
shows that $C$ is generated by $C_0$ and $C_1$. Since
$\Bbbk$ is algebraically closed, it forces that $C_0
=\Bbbk^{\oplus d}$. This means that $C$ is standard.

The consequence follows from Theorem \ref{xxthm0.2}(1)
and Definition \ref{xxdef3.2}.
\end{proof}

By the above lemma, we can apply Theorems \ref{xxthm0.1}
and \ref{xxthm0.2} to the situation where $R$ is a 
preprojective algebra as studied by Weispfenning \cite{We}.


\vspace{0.5cm}

\end{document}